\documentclass[twoside,11pt]{amsart} 
\usepackage{amsfonts,amssymb,latexsym,amsthm, epsfig,amsmath,oldgerm,amscd}
\usepackage[all]{xy}
\def\Com{\UseComputerModernTips}

\newcommand{\mf}{\mathfrak}

\newcommand{\ra}{\rightarrow}
\newcommand{\Ra}{\Rightarrow}

\newcommand{\mbb}{\mathbb}
\newcommand{\tn}{\textnormal}

\newtheorem{de}{Definition}[section]
\newtheorem{re}[de]{Remark}
\newtheorem{pr}[de]{Proposition} 
\newtheorem{tr}[de]{Theorem}
\newtheorem{lm}[de]{Lemma}

\newtheorem{co}[de]{Corollary}

\def\vp{\rm \vspace{0.2cm}}

\def\GL{\rm GL}

\def\SL{\rm SL}

\def\EO{\rm EO}

\def\SO{\rm SO}
\def\E{\rm E}

\def\T{\rm T}
\def\ET{\rm ET}
\def\G{\rm G}
\def\Sp{\rm Sp}

\def\I{\rm I}
\def\T{\rm T}
\def\O{\rm O}
\def\ESp{\rm ESp}
\def\EO{\rm EO}

\def\es{\rm S}

\begin{document}
%\linespread{2}
\title{
Quillen--Suslin Theory for Classical 
Groups: Revisited over Graded Rings}
\author{Rabeya Basu}
\thanks{Corresponding Author: Rabeya Basu -- rabeya.basu@gmail.com}
\author{Manish Kumar Singh}

\date{}

\maketitle


\noindent
{\small Abstract:  In this paper we deduce a graded version of Quillen--Suslin's Local-Global Principle for the traditional classical groups, {\it viz.} general linear, 
symplectic and orthogonal groups and establish its equivalence of the normality property of the respective elementary subgroups. This  
generalizes previous result of Basu--Rao--Khanna; {\it cf.} \cite{rrr}. Then, as an application, 
we establish an analogue Local-Global Principle for the commutator subgroups of the special linear and symplectic groups. 
Finally, by using Swan-Weibel's homotopy trick, we establish 
graded analogue of the Local-Global Principle 
for the transvection subgroups of the full automorphism groups of the linear, symplectic and orthogonal modules. This generalizes 
the previous result of Bak--Basu--Rao; {\it cf.} \cite{BBR}.}

\section{Introduction} 

In \cite{rrr}, the first author with Ravi A. Rao and Reema Khanna revisited Quillen-Suslin's Local-Global Principle 
for the linear, symplectic and orthogonal groups to show the equivalence of the 
Local-Global Principle and Normality property of their Elementary subgroups. A ``relative version'' 
is also established recently by the same authors; {\it cf. \cite{RB2}}. In \cite{rrr}, the authors have also deduced a Local-Global Principle for 
the commutator subgroups of the special linear groups, and symplectic groups. In this article we aim to revisit those results over commutative $\mbb{N}$-graded
rings with identity. We follow the line of proof described in \cite{rrr}.

%In \cite{vorst}, T. Vorst proved that for two special types of regular rings the special linear groups over the polynomial rings coincide 
%with their elementary subgroups; {\it i.e.} 
%$${\SL}_n(R[X])={\E}_n(R[X]) \,\, {\rm for }\; n \ge 3.$$ The symplectic and orthogonal analogues of Vorst's result was deduced in \cite{RB1} (unpublished). 
%Using that, in \cite{rrr}, it has been shown that one gets variant of the local-global principle for the commutator subgroups 
%for above three types of classical groups. 

For the linear case, the graded version of the Local-Global Principle was studied by Chouinard in \cite{C}, and by J. Gubeladze in \cite{gubel}, \cite{gubel2}. 
Though the analogue results are expected for the symplectic and orthogonal groups, according to our best knowledge, 
that is not written explicitly in any existing literatures. By deducing the equivalence, we establish the graded version of the 
Local-Global Principle for the above two types of classical groups. 

To generalize the existing results, from the polynomial rings to the graded rings, one has to use the line of proof of Quillen--Suslin's patching argument 
({\it cf. \cite{QUI}, \cite{SUS}}), 
and Swan--Weibel's homotopy trick. 
For a nice exposition we refer to \cite{gubel} by J. Gubeladze. 

Though, we are writing this article for commutative graded rings, one may also 
consider standard graded algebras which are finite over its center. For a commutative $\mbb{N}$-graded ring with 1, 
we establish:

{\bf $(1)$ (Theorem \ref{thm 22}}) Normality of the elementary subgroups is equivalent to the graded 
Local-Global Principle for the linear, symplectic and orthogonal groups. 

{\bf $(2)$ (Theorem \ref{thm 21})} Analogue of Quillen--Suslin's Local-Global 
Principle for the commutator subgroups of the special linear and symplectic groups over graded rings.

In \cite{BBR}, the following was established by the first author with 
A. Bak and R.A. Rao: 

{\bf (Local-Global Principle for the Transvection Subgroups})
An analogue
of Quillen--Suslin's Local-Global Principle for the transvection subgroup of 
the automorphism group of projective, 
symplectic and orthogonal modules of global rank at least 1 and 
local rank at least 3, under the assumption that the projective module has 
constant local rank and that the symplectic and orthogonal modules are locally 
an orthogonal sum of a constant number of hyperbolic planes. 

In this article, we observe that by using Swan-Weibel's homotopy trick,
one gets an analogue statement for the graded case. More precisely, we deduce the following fact:

{\bf $(3)$ (Theorem \ref{thm 51})} \label{tr 31} Let $A = \bigoplus_{i = 0 }^{\infty} A_i$ be a graded ring and $Q \simeq P\oplus A_0$ be a 
projective $A_0$-module ($Q \simeq P \oplus\mathbb{H}(A_0)$ for symplectic and orthogonal modules). If an automorphism of $Q \otimes_{A_0} A$ 
is locally in the transvection subgroup, then it is globally in the transvection subgroup.

%{\bf Result 3 (Theorem \ref{thm 32})} For a graded ring $A$ with $A_0$ 
%regular local essentially of finite type over a field $k$, 
%$${\SL}_n(A)={\E}_n(A)~~ {\rm  for  } ~~n\ge 3, ~~{\rm and }$$  
%$${\Sp}_{2m}(A)={\ESp}_{2m}(A) ~~{\rm  for  }~~ m\ge 3.$$
%
%
%{\bf Result 4 (Theorem \ref{thm 41})} For a geometrically regular ring containing a field 
%the commutator subgroup of linear, symplectic and orthogonal groups of the 
%graded ring satisfies local-global principle. 

\section{Definitions and Notations}

Let us start by recalling the following well-known
fact: Given a ring $R$ and a subring $R_0$, one can express $R$ as a
direct limit of  subrings which are finitely generated over $R_0$ as
rings. Considering $R_0$ to be the minimal subring ({\it i.e.} the
image of $\mbb{Z}$), it  follows that every ring is a direct limit of
Noetherian rings. Hence we may consider $R$ to be Noetherian ({\it cf.} Pg 271 \cite{MM}).

Throughout this paper we assume $A$ to be a Noetherian, commutative graded ring with identity $1$. We shall write 
$A = A_0\oplus A_1 \oplus A_2 \oplus \cdots$. As we know the multiplication in a graded ring satisfies the following property: 
For all $i, j$, $A_i A_j \subset A_{i+j}$. An element $a \in A$ will be denoted by $a = a_0 + a_1 + a_2 + \cdots $, where $a_i\in A_i$ for each $i$, and 
and all but finitely many $a_i$ are zero. Let $A_+= A_1 \oplus A_2 \oplus \cdots$. 
Graded structure of $A$ induces a graded structure on ${\rm M}_n(A)$ (ring of $n \times n$ matrices). 
Let $S$ be a multiplicatively closed subset of $A_0$. 
Then  for a non-zero divisor $s \in S$ we shall denote the 
localization of a matrix $\alpha \in {\rm M}_n(A)$ to ${\rm M}_n(A_s)$ as $\alpha_s$. 
Otherwise, $\alpha_i,\; i \in \mathbb{Z}$ will represent the $i$-th component of $\alpha$. 
We shall use standard notations ${\GL}_n(A)$ and ${\SL}_n(A)$ to denote the  group of invertible matrices and its subgroup of all 
invertible matrices with determinant $1$ respectively.

We recall the well-known ``{\bf Swan-Weibel's homotopy trick}'', which is the main ingredient to handle the graded case.  

\begin{de}
Let $a \in A_0$ be a fixed element. We fix an element $b = b_0 + b_1 + \cdots$ in $A$ and define a ring homomorphism 
$\epsilon: A \rightarrow A[X]$ given by \[ \epsilon(b) = \epsilon (b_0 + b_1 + \cdots )\; = \; b_0 + b_1X + b_2X^2 + \cdots + b_iX^i + \cdots .\]
%then we define $$\epsilon_{a}: b \mapsto b^+(a) $$ as
%$$\epsilon_a(b) = \epsilon_{\mid a}(b) = b_0 + b_1a + b_2 a^2 + \cdots + b_i a^i \cdots .$$
%We extend this definition to ${\rm M}_n(A)$ in a similar way.

Then we evaluate the polynomial $\epsilon(b)(X)$ at $X = a$ and denote the image by $b^+(a)$ {\it i.e.} $b^+(a) = \epsilon(b)(a)$.
Note that $\big(b^+(x)\big)^+(y) = b^+(xy)$. Observe, $b_0=b^{+}(0)$. 
We shall use this fact frequently.
\end{de}

The above ring homomorphism $\epsilon$ induces a group homomorphism at the ${\GL}_n(A)$ level for every $n \geq 1$, {\it i.e.} for 
$\alpha \in {\GL}_n(A)$ we get a map $$\epsilon: {\GL}_n(A) \rightarrow {\GL}_n(A[X]) \text{ defined by} $$ 
$$\alpha = \alpha_0 \oplus \alpha_1 \oplus \alpha_2 \oplus \cdots \mapsto \alpha_0 \oplus \alpha_1X \oplus \alpha_2X^2 \cdots,$$
where $\alpha_i\in {\GL}_n(A_i)$. 
As above for $a \in A_0$, we define $\alpha^+(a)$ as $$\alpha^+(a) = \epsilon(\alpha)(a).$$

Now we are going to recall the definitions of the traditional classical groups, {\it viz.} the general linear groups, the symplectic and orthogonal groups 
(of even size) and their 
type subgroups, {\it viz.} elementary (symplectic and orthogonal elementary resp.) subgroups.

Let $e_{ij}$ be the matrix with $1$ in the $ij$-position and $0$'s elsewhere.
The matrices of the form $\{ {\E}_{ij}(\lambda) : \lambda \in A \mid i \neq j\},$  where  $$ {\E}_{ij} (\lambda) = {\I}_n + \lambda e_{ij}$$  
 are called the \textbf{elementary generators}.

\begin{de}
The subgroup generated by the set $\{ {\E}_{ij}(\lambda) : \lambda \in A \mid i \neq j\}$, is called the 
\textbf{elementary subgroup of the general linear group} and is denoted by ${\E}_n(A)$. 
Observe that ${\E}_n(A) \subseteq {\SL}_n(A)\subseteq {\GL}_n(A)$. 
\end{de}

Let $\sigma$ be a permutation defined by: For $i \in \{1, \ldots, 2m \}$
$$\sigma(2i) = 2i -1 \text{ and }\sigma(2i-1) = 2i.$$
With respect to this permutation, we define two $2m \times 2m$ forms (viz.) $\psi_m$ and $\widetilde{\psi}_m$  as follows: For $m>1$, let 
\[ \psi_1 = 
\begin{bmatrix}
    0  &  1      \\
    -1  &  0      
\end{bmatrix}
\text{ and for $m>1$},\,\, \psi_m =
\begin{bmatrix}
    \psi_{m-1}  &  0      \\
    0  &  {\I}_2     
\end{bmatrix},
\]

\[ \widetilde{\psi}_1 = 
\begin{bmatrix}
    0  &  1      \\
    1  &  0      
\end{bmatrix}
\text{ and for $m>1$}, \,\,\widetilde{\psi}_m =
\begin{bmatrix}
    \widetilde{\psi}_{m-1}  &  0      \\
    0  &  {\I}_2     
\end{bmatrix}. 
\]

Using above two forms, we define the following traditional classical groups:

\begin{de} A matrix $\alpha\in {\GL}_{2m}(A)$ is called symplectic if it fixes $\psi_m$ under the action of conjugation, {\it i.e.}
$$ \alpha^t \psi_{m} \alpha = \psi_{m}.$$
The group generated by all symplectic matrices is called the \textbf{symplectic group} and is denoted by ${\Sp}_{2m}(A)={\Sp}_n(A)$, where $n = 2m$. 
\end{de}

\begin{de} A matrix $\alpha \in {\GL}_{2m}(A)$ is called orthogonal if it fixes $\widetilde{\psi}_m$ under the action of conjugation, {\it i.e.}
$$ \alpha^t \widetilde{\psi}_{m} \alpha = \widetilde{\psi}_{m}.$$
The group generated by all orthogonal matrices is called \textbf{orthogonal group} and is denoted by ${\O}_{2m}(A)={\O}_{n}(A)$, where $n = 2m$.
\end{de}
\begin{de}
The matrices of the form $$\{{se}_{ij}(z)\in {\GL}_{2m}(A) : z \in A \mid i \neq j\},$$  where $$se_{ij}(z) = {\I}_{2m} + ze_{ij} \text{ if } i = \sigma(j)$$ or
$$se_{ij}(z) = {\I}_{2m} + ze_{ij} - (-1)^{i+j}ze_{\sigma(j) \sigma(i)} \text{ if } i \neq \sigma(j) \text{ and } i < j.$$    
are called \textbf{ symplectic elementary generators}.
The subgroup generated by symplectic elementary generators is called \textbf{symplectic elementary group}.
\end{de}
\begin{de}
The matrices of the form $\{{oe}_{ij}(z) : z \in A \mid i \neq j\},$ where 
$$oe_{ij}(z) = {\I}_{2m} + ze_{ij} - ze_{\sigma(j) \sigma(i)}, \text{ if } i \neq \sigma(j) \text{ and } i < j.$$    
are called \textbf{orthogonal elementary generators}.

The subgroup generated by orthogonal elementary generators is called \textbf{orthogonal elementary group}.

\begin{re} \tn{It is a well known fact that the elementary subgroups are normal subgroups of the respective classical groups; for $n\ge 3$ in the 
linear case, $n=2m\ge 4$ in the symplectic case, and $n=2m\ge 6$ in the orthogonal case. 
The linear case was due to A. Suslin ({\it cf.}  \cite{Tu}), symplectic was proved by V. Kopeiko ({\it cf.} \cite{Tu}). 
Finally the orthogonal case was proved by Suslin--Kopeiko ({\it cf.}  \cite{SUSK}).}
\end{re}

\end{de} We use the following notations
to treat above three groups uniformly: 
 $${\G}( n, A) \text{ will denote } {\GL}_n(A) \text{ or }  {\Sp}_{2m}(A) \text{ or }{\O}_{2m}(A).$$ 
$${\E}(n, A) \text{ will denote }  {\E}_n(A)  \text{ or } {\EO}_{2m}(A) \text{ or } 
{\ESp}_{2m}(A).$$ 
$${\rm S}(n, A) \text{ will denote } {\SL}_n(A) \text{ or } {\SO}_{2m}(A), \text{ and }$$ 
$$ ge_{ij} \text{ will denote the elementary generators of } {\E}_n(A) \text{ or } {\ESp}_n(A) \text {or } {\EO}_n(A).$$

Here ${\SO}_{n}(A)$ is the subgroup of ${\O}_n(A)$ with determinant $1$. Throughout the paper we will assume $n = 2m$ 
to treat the non-linear cases. We shall assume $n\ge 3$ while treating the linear case, and 
$n\ge 6$; {\it i.e.} $m\ge 3$ while treating the symplectic and orthogonal cases. 
Let \[ \widetilde{v} =
\begin{cases}
v \psi_m  \text{ in the symplectic case, and } \\
v \widetilde{\psi}_m \text{ in the orthogonal case. }
\end{cases}
\]
\begin{de}
Let $v, w \in A^n$ be vectors of length $n$, then we define the inner product $\langle v, w \rangle $ as follows:
\begin{enumerate}
\item $\langle v, w \rangle $  = $v^tw$ in the linear case,
\item $\langle v, w \rangle $  = $\widetilde{v} w$ in the symplectic and orthogonal cases.
\end{enumerate}
\end{de}

\begin{de}Let $v, w \in A$, we define the map $M: A^n \times A^n \rightarrow M_n(A)$ as follows:
\begin{enumerate}
\item $M(v,w) = vw^t$ in the linear case,
\item $M(v,w) = v\widetilde{w} + w\widetilde{v}$ in the symplectic case,
\item $M(v,w) = v\widetilde{w} - w\widetilde{v}$ in the orthogonal case.
\end{enumerate}
\end{de}

By $[G, G]$, we denote the commutator subgroup of $G$ {\it i.e.} group generated by elements of the form $ghg^{-1}h^{-1}$, for $g,h\in G$.

\begin{de} A row $(v_1, \ldots, v_n) \in A^n$ is called \textbf{unimodular row of length $n$} 
if there exist $b_1,\ldots,b_n \in A$ such that 
$\Pi_{i=1}^n v_ib_i = 1$; {\it i.e.} If $I$ is the ideal generated by 
$\langle v_1, \ldots, v_n \rangle$, then $(v_1, \ldots, v_n)$ is unimodular 
if and only if $I = A$. In the above case we say $A$ has $n$-cover.
\end{de}

\section{Local-Global Principle for classical groups}

Before discussing the main theorem we recall following standard facts.

\begin{lm}[Splitting Lemma] \label{lemm 16} The elementary generators of $n \times n$ 
matrices satisfy the following property: 
$$ge_{ij}(x + y) = ge_{ij}(x)\; ge_{ij}(y)$$ for all $x, y \in A$ 
and for all $i, j =1,\ldots,n$ with $i\ne j$.
\end{lm}

\begin{proof} Standard. 
({\it cf.}  \cite{rrr}, \$3, Lemma $3.2$).
\end{proof}

\begin{lm} \label{adjust} 
 Let $G$ be a group, and $a_i, b_i\in G$ for $i=1,\ldots,r$. 
 Let $J_k = \underset{j=1}{\overset{k}\Pi} a_j$.
 Then $$\underset{i=1}{\overset{r}{\huge\Pi}} a_ib_i = \underset{i=1}{\overset{r}\Pi} J_ib_iJ_i^{-1} 
 \underset{i=1}{\overset{r}\Pi} a_i.$$ 
\end{lm}

\begin{proof}  {\it cf.} \cite{rrr}, \$3, Lemma $3.4.$

\end{proof}

Following structural lemma pays a key role in the proof of main theorem. 
It is well-known for the polynomial rings; {\it cf.} \cite{rrr} (\$3, Lemma $3.6$).
Here, we deduce the analogue for the graded rings. 
\begin{lm}\label{lemm 11} Let ${\G}(n, A, A_+)$ denote the subgroup of ${\G}(n,A)$ which 
is equal to ${\I}_n$ modulo $A_+$. 
Then the group ${\G}(n, A, A_+) \cap {\E}(n, A)$ is generated by the elements of the type 
$\epsilon\, ge_{ij}(A_+)\,\epsilon^{-1}$ for some $\epsilon \in {\E}(n,A_0)$.
\end{lm}

\begin{proof}
Let $\alpha \in {\E}(n, A) \cap {\G}(n, A, A_+)$. Then  we can write 
$$\alpha = \underset{k=1}{\overset{r}\Pi} ge_{i_kj_k}(a_k)$$ 
for some elements $a_k \in A$, $k=1,\ldots,r$. 
As $a_k = ({a_0})_k + ({a_+})_k$ for some $(a_0)_k\in A_0$ and $(a_{+})_k\in A_{+}$. Using the splitting lemma (Lemma \ref{lemm 16}) we can rewrite the expression as:
$$\alpha =  \underset{k=1}{\overset{r}\Pi}\big(ge_{i_k j_k}{(a_0)}_k\big) \, \big(ge_{i_k j_k}{(a_+)}_k\big).$$
We write $\epsilon_t = \Pi_{k=1}^{t}ge_{i_k j_k}\big({(a_0)}_k\big)$ for $t \in\{1,2,\ldots,r\}$. Observe that $\epsilon_r = {\I}_n$, as $\alpha \in {\G}(n, A, A_+)$.
Then 
$$\alpha = \Big(\underset{k=1}{\overset{r}\Pi} \epsilon_k ge_{i_k j_k}\big({(a_+)}_k\big) 
\epsilon_k^{-1}\Big) \Big(\underset{k=1}{\overset{r}\Pi}ge_{i_k j_k}\big({(a_0)}_k\big)\Big) = 
\mathcal{A}\mathcal{B} \text{ (say) },$$ 
where $\mathcal{A} = \Big(\underset{k=1}{\overset{r}\Pi} \epsilon_k ge_{i_k, j_k}\big({(a_+)}_k\big) \epsilon_k^{-1}\Big) $ 
and $\mathcal{B} = \Big(\underset{k=1}{\overset{r}\Pi}ge_{i_k j_k}\big({(a_0)}_k\big)\Big). $ 
Now go modulo $A_+$. Let bar (--) denote the quotient ring modulo $A_{+}$. Then, 
$$\overline{\alpha} = \overline{{\I}}_n = 
\overline{\mathcal{A}} \overline{\mathcal{B}} = \overline{{\I}}_n \overline{\mathcal{B}} =
\bar{{\I}}_n \implies \overline{\mathcal{B}} = \overline{{\I}}_n,$$ 
as $\alpha \in {\G}(n, A, A_+)$. Since entries of $\mathcal{B}$ 
are in $A_0$, it follows that $\mathcal{B} = {\I}_n$. 
$$\alpha =\underset{k=1}{\overset{r}\Pi} \epsilon_k \,\big(ge_{i_k j_k}\big({(a_+)}_k\big)\big) \,\epsilon_k^{-1};$$ 
as desired.
\end{proof}

Now we prove a variant of ``Dilation Lemma'' mentioned in the statement of $(3)$ of Theorem \ref{thm 22}.
\begin{lm} \label{lemm 12} Assume the ``Dilation Lemma'' $($Theorem  {\rm \ref{thm 22} - $($\ref{Prop 13}$)$}$)$ 
to be true. Let  $\alpha_s \in {\E}(n, A_s),\;$ with $\alpha^+(0) = {{\I}_n}$. Then one gets  $$ \alpha^+(b + d)\alpha^+(d)^{-1} \in {\E}(n, A)$$
for some $s, d\in A_0$ and $b = s^l, l \gg 0$. 
\end{lm}

\begin{proof} We have $\alpha_s \in {\E}(n, A_s)$.
Hence $\alpha_s^+(X) \in {\E}(n, A_s[X])$. Let $\beta^{+}(X) = \alpha^+(X + d){\alpha^+(d)}^{-1}$, where $d\in A_0$. 
Then $\beta^{+}_s(X) \in {\E}(n, A_s[X])$ and $\beta^+(0) = {{\I}_n}$.
Hence by Theorem \ref{thm 22} - (\ref{Prop 13}) their exists 
$\widetilde{\beta}(X) \in {\E}(n, A[X])$ such that $\widetilde{\beta}_s(X) = \beta_s^+(bX)$. Putting $X = 1$, we get the required result.
\end{proof}

Following is a very crucial result we need for our method of proof. There are many places we use this fact in a very subtle way. 
In particular, we mainly use this lemma for the step $(4)\Ra(3)$ of \ref{thm 22}. 

\begin{lm} {\rm ({\it cf.}\cite{HV}, Lemma 5.1)} \label{lemm 13}  Let $R$ be a Noetherian ring and $s\in R$.
Then there exists a natural number  $k$ such that the homomorphism
${\G}(n,s^kR) \ra {\G}(n, R_s)$  $($induced by localization
homomorphism $R \ra R_s)$ is injective.  Moreover, it follows that
the induced map  ${\E}(n,R,s^kR) \ra {\E}(n,R_s)$ is
injective. 
\end{lm}

The next two lemma's will be used in intermediaries to prove the equivalent conditions mentioned in Theorem \ref{thm 22}. We  state it without proof.
\begin{lm}\label{lemm 14} Let $v = (v_1, v_2, \ldots, v_n)$ be a unimodular row over over a 
commutative semilocal ring $R$. Then the row $(v_1, v_2, \ldots, v_n)$ is completable; {\it i.e.}  $(v_1, v_2, \ldots, v_n)$ 
elementarily equivalent to the row $(1,0, \ldots, 0)$; 
{\it i.e.} their exists $\epsilon \in {\E}_n(R)$ such that $(v_1, v_2, \ldots, v_n)\epsilon = (1,0, \ldots, 0)$.
\end{lm}

\begin{proof}
{\it cf.}  \cite{RB1}, Lemma $1.2.21$.
\end{proof}

\begin{lm}\label{lemm 15} Let $A$ be a ring and $v \in {\E}(n, A)e_1$. Let $w \in A^n$ be a 
column vector such that $\langle v, w \rangle = 0$. Then ${\I}_n + M(v, w) \in E(n, A)$.
\end{lm}

\begin{proof}
{\it cf.} \cite{RB1}, Lemma $2.2.4$.
\end{proof}
\begin{tr}
\label{thm 22} The followings are equivalent for any graded ring $A=\oplus_{i = 0 }^{\infty} A_i$ for
$n \geq 3$ in the linear case and $n \geq 6$ otherwise.

\begin{enumerate}

\item\label{Prop 11} \textbf{\rm{(Normality)}}:
${\E}(n, A)$ is a normal subgroup of ${\G}(n, A)$.

\item\label{Prop 16}
If $v \in {\rm Um}_n(A)$ and $\langle v, w \rangle = 0$, then ${\I}_n + M(v, w) \in {\E}(n, A)$.

\item\label{Prop 12} \textbf{\rm{(Local-Global Principle)}}:
Let $\alpha \in {\G}(n, A)$ with $\alpha^+(0) = {\I}_n$. If 
for every maximal ideal $\mathfrak{m}$ of $A_0$
$\alpha_\mathfrak{m} \in {\E}(n, A_\mathfrak{m})$,  
then $\alpha \in {\E}(n, A)$.

\item\label{Prop 13} \textbf{\rm{(Dilation Lemma)}}: 
Let $\alpha \in {\G}(n, A)$ with $\alpha^+(0) = {\I}_n$ and $\alpha_s \in {\E}(n,A_s)$ for some non-zero divisor $s \in A_0$.
Then there exists $\beta\in {\E}(n,A)$ such that 
$$\beta_s^{+}(b)=\alpha_s^{+}(b)$$ for some $b = s^l$; $l \gg 0$. {\it i.e.} 
$\alpha^+_s(s^l)$ will be defined over $A$ for $l \gg 0$. 

\item \label{Prop 14} If $\alpha \in {\E}(n,A)$, then $\alpha^+(a) \in {\E}(n,A)$, for every $a \in A_0$.

\item \label{Prop 15}   If $v \in {\E}(n, A)e_1,$ and $\langle v, w \rangle = 0$, then ${\I}_n + M(v,w) \in {\E}(n, A)$.

\end{enumerate}
\end{tr}
\begin{proof}

\noindent (\ref{Prop 15}) $\Rightarrow$ (\ref{Prop 14}):
Let $\alpha  = \underset{k=1}{\overset{t}\Pi} \big({\I}_n + a\;M(e_{i_k}, e_{j_k})\big) \text{ where } a \in A$, and $t\ge 1$, a positive integer. 
Then $$\alpha^+(b) =\underset{k=1}{\overset{t}\Pi} \big( {\I}_n + a^+(b)\;M(e_{i_k}, e_{j_k})\big),$$ where 
$b \in A_0$. Take $v = e_i$ and $w = a^+(b)e_j$, Then $\alpha = {\I}_n + M(v,w)$ and $\langle v, w \rangle = 0$. Indeed,

Linear case:
\begin{align*}
{\rm M}(v, w) = \; & v w^t =  e_i(a^+(b)e_j)^t = a^+(b) e_i e_j^t = a^+(b)\;{\rm M}(e_i, e_j), \\
\langle v, w\rangle = \;  &v^tw =  e_i^t(a^+(b)e_j)^t  =  a^+(b)\; e_i^t e_j =  a^+(b) \langle e_i, e_j \rangle = 0. 
\end{align*}

Symplectic case:\newline
\begin{align*}
{\rm M}(v, w) = \; & v w^t\psi_m +  w v^t\psi_m =   e_i (a^+(b)e_j)^t\psi_m + a^+(b)e_j e_i^t\psi_m \\
= \; & a^+(b) e_i e_j^t\psi_m + a^+(b)e_j e_i^t\psi_m =  a^+(b)\;{\rm M}(e_i, e_j), \\
\end{align*}
\begin{align*}
\langle e_i, e_j \rangle =  e_i^t\psi_me_j = (e_i^t\psi_m e_i)(e_i^te_j) = \psi_m (e_i^te_j) = 0 \text{ hence, } \\
\langle v, w\rangle =  \; v^t\psi_mw =  e_i^t\psi_m a^+(b)e_j =  a^+(b)e_i^t e_j =  a^+(b)\langle e_i, e_j \rangle = 0.
\end{align*}

Orthogonal case:\newline
\begin{align*}
{\rm M}(v, w) = \; & v.w^t\widetilde{\psi}_m -  w.v^t\widetilde{\psi}_m =  e_i.(a^+(b)e_j)^t\widetilde{\psi}_m - a^+(b)e_j.e_i^t\widetilde{\psi}_m \\
= \; & a^+(b) e_i.e_j^t\widetilde{\psi}_m - a^+(b)e_j.e_i^t\widetilde{\psi}_m =  a^+(b)\;{\rm M}(e_i, e_j),
\end{align*}
\begin{align*}
\langle e_i, e_j \rangle =  e_i^t\widetilde{\psi}_me_j = (e_i^t\widetilde{\psi}_m e_i)(e_i^te_j) = \widetilde{\psi}_m (e_i^te_j) = 0  \text{ hence, } \\
\langle v, w\rangle = \; v^t\widetilde{\psi}_mw =  e_i^t\widetilde{\psi}_ma^+(b)e_j =  a^+(b)e_i^t e_j =  a^+(b)\langle e_i, e_j \rangle = 0.
\end{align*}

Hence applying (\ref{Prop 15}) over ring $A$, we have $\alpha^+(b) \in {\E}(n, A)$. Therefore $\alpha^+(b) \in {\E}(n, A)$ for $b \in A_0$.

\noindent (\ref{Prop 14}) $\Rightarrow$  (\ref{Prop 13}):
Since $\alpha_s \in {\E}(n, A_s)$ with  $({\alpha_0}_s) = {\I}_n$, the diagonal entries of $\alpha$ 
are of the form $1 + g_{ii}$, where $g_{ii} \in ({A_+})_s$ and off-diagonal entries are of the 
form $g_{ij}$, where $i\ne j$ and $g_{ij} \in ({A_+)}_s$.

We choose $l$ to be large enough such that $s^l$ is greater than the common denominator of all 
$g_{ii}$ and $g_{ij}$. Then using  (\ref{Prop 14}), we get $$\alpha_s^+(s^l) \in {\E}(n, A_s).$$ 
Since that $\alpha^+(s^l)$ permits a natural pullback (as denominators are cleared), we have $\alpha^+(s^l) \in {\E}(n, A)$.

\noindent (\ref{Prop 13}) $\Rightarrow$  (\ref{Prop 12}):

Since $\alpha_\mathfrak{m} \in {\E}(n ,A_\mathfrak{m})$, we have an element $s \in A_0 - \mathfrak{m}$ such that 
$\alpha_s \in {\E}(n, A_s)$. Let $s_1, s_2, \ldots, s_r \in A_0$ be non-zero divisors with $s_i \in A_0 - \mathfrak{m}_i$ 
such that $\langle s_1, s_2, \ldots, s_r \rangle = A$. From  (\ref{Prop 13}) we have $\alpha^+(b_i) \in {\E}(n, A),$ 
for some $b_i = {{s_i}^{l_i}}$ with $ b_1 + \cdots + b_r  = 1 $. Now consider $\alpha_{s_1s_2\ldots s_r}$, 
which is the image of $\alpha$ in $A_{s_1 \cdots s_r}$. Due to Lemma \ref{lemm 13}, $\alpha \mapsto \alpha_{s_1s_2 \cdots  s_r}$ 
is injective and hence we can perform our calculation in $A_{s_1 \cdots s_r}$ and then pull it back to $A$.
\begin{align*}
 \alpha_{s_1s_2\ldots s_r} = \; & \alpha_{s_1s_2\ldots s_r}^+(b_1 + b_2 \cdots + b_r) \\ = 
 \; & {{\big((\alpha_{s_1})}_{s_2s_3\ldots}}\big)^+(b_1 + \cdots+ b_r) {{\big((\alpha_{s_1})}_{s_2s_3\ldots}}\big)^+(b_2 + \cdots + b_r)^{-1} \cdots \\
 & {\big((\alpha_{s_i})}_{s_1 \ldots \hat{s_i} \ldots s_r}\big)^+(b_i + \cdots + b_r)
 {\big((\alpha_{s_i})}_{s_1 \ldots \hat{s_i} \ldots s_r}\big)^+(b_{i+1} + \cdots + b_r)^{-1}\\ 
 & \cdots {{\big((\alpha_{s_r})}_{s_1s_2\ldots s_{r-1}}}\big)^+(b_r){{\big((\alpha_{s_r})}_{s_1s_2\ldots s_{r-1}}}\big)^+(0)^{-1}
\end{align*} 
Observing that each $${\big((\alpha_{s_i})}_{s_1s_2 \ldots \hat{s_i} \ldots s_r}\big)^+
(b_i + \cdots + b_r){\big((\alpha_{s_i})}_{s_1 \ldots \hat{s_i} \ldots s_r}\big)^+(b_{i+1} + \cdots + b_r)^{-1} \in {\E}(n, A)$$ 
due to Lemma \ref{lemm 12} (here $\hat{s}_i$ means we omit $s_i$ in the product $s_1\cdots \hat{s_i}\cdots s_r$), 
we have $\alpha_{s_1 \cdots s_r} \in {\E}(n, A_{s_1 \cdots s_r})$ and hence $\alpha \in {\E}(n, A)$.
\vspace{0.2em}

\begin{re}\label{rmk 001}
Following is a commutative diagram (here we are assuming $\langle s_i, s_j \rangle = A$):

\Com
$$\xymatrix{ A  \ar[r] \ar[d] & A_{s_i} \ar[d]\\
 A_{s_j}\ar[r] & A_{s_i s_j} } $$

Let $\theta_i = \alpha_{s_i}^+(b_i + \cdots + b_r)\alpha_{s_i}^+(b_{i-1}+\cdots+b_r) \in A_{s_i}$ and $\theta_{ij}$ 
be the image of $\theta_{i}$ in $A_{s_is_j}$ and similarly $\pi_j = \alpha_{s_j}^+(b_j + \cdots + b_r)\alpha_{s_j}^+(b_{j-1}+\cdots+b_r) \in A_{s_j}$ 
and $\pi_{s_i,s_j}$ be its image in $A_{s_i s_j}$. Then due to Lemma \ref{lemm 13} the product $\big(\theta_{s_is_j}\big)\big(\pi_{s_is_j}\big)$ 
can be identified with the product $\theta_i\pi_j$.
\end{re}
\iffalse

\[
\begin{tikzcd}
A\arrow[hookrightarrow]{r} \arrow[hookrightarrow]{d}
  & A_{s_i} \arrow[hookrightarrow]{d} \\
A_{s_j}\arrow[hookrightarrow]{r}
  & A_{s_i s_j}
\end{tikzcd}
\]
\fi

\noindent (\ref{Prop 12}) $\Rightarrow$ (\ref{Prop 16}): 
Since polynomial rings are special case of graded rings, the result follows by using  $3 \Rightarrow 2$ in \cite{rrr}, $\S 3$.\vp

\noindent (\ref{Prop 16}) $\Rightarrow$ (\ref{Prop 11}) $\Rightarrow$ (\ref{Prop 15}):
The proof goes as in \cite{rrr}. (\$$3$, $2 \Rightarrow 1 \Rightarrow 7 \Rightarrow 6$). 

\end{proof}

\section{Local-Global Principle for commutator subgroup}
  
In this section we deduce the analogue of Local-Global Principle for the commutator subgroup of the linear and the symplectic group over graded rings. 
Unless mentioned otherwise, we assume $n \ge 3$ for the linear case and $n \geq 6$ for the symplectic case. \vp 

Let us begin with the following well-known fact for semilocal rings.
\begin{lm}\label{lemm 26} Let $A$ be a semilocal commutative ring with identity. Then for $n \geq 2$ in the linear case and 
$n \geq 4$ in the symplectic case, one gets $${\es}(n, A) = {\E}(n, A).$$
\end{lm}

\begin{proof}
\textit{cf.} Lemma $1.2.25$ in \cite{RB1}.
\end{proof}

\begin{re}
If $\alpha = (\alpha_{ij}/1) \in {\G}(n, A_s)$, then $\alpha$ has a natural pullback $\beta = (\alpha_{ij}) \in {\G} (n, A)$ such that $\beta_s = \alpha$.
If $\alpha_s \in {\es}(n, A_s)$ such that it admits a natural pullback $\beta \in {\G}(n, A)$,
then $\beta \in {\es}(n, A)$.
\end{re}

The next lemma deduces an analogue result of ``Dilation Lemma'' (Theorem \ref{thm 22}-($3$)) \underline{ for ${\es}(n, A)$}; the special linear (resp. symplectic) group.
\begin{lm}\label{lemm 24} Let $\alpha = (\alpha)_{ij} \in {\es}(n, A)$. Then $\alpha^+(a) \in {\es}(n, A),$ 
where $a \in A_0$. Hence if $\alpha \in G(n, A)$ with $\alpha^+(0) = {{\I}_n}$ and $\alpha_s \in {\es}(n, A_s) \text{ then } 
\alpha_s^+(b) \in {\es}(n, A)$ $($after identifying $\alpha_s^+(b)$ with its pullback by using Lemma \ref{lemm 13}$)$, 
where $s \in A_0$ is a non-zero divisor and $b = s^l$ with $l \gg 0$, for $n \geq 1$ in linear case and $n \geq 2$ in symplectic case. 
\end{lm}

\begin{proof}
Since $\text{ det }:A \rightarrow A^* \text{ (units of } A) \subset A_0,$ and $A_iA_j \subset A_{i+j}$, 
the non-zero component of $\alpha$ doesn't contribute for the value of the determinant and hence $$\text{det }\alpha =  
\text{det }\alpha^+(0).$$ Therefore, if $\beta = \alpha^+(a)$, then $$\text{det }\beta = \text{det }\beta^+(0) =  
\text{det }\big( \alpha^+(a)\big)^+(0) =  \text{det }\alpha^+(a \cdot 0) = \text{det }\alpha^+(0) = 1.$$ Hence $\alpha^+(a) \in {\es}(n, A)$.

Since $\alpha_s \in {\es}(n, A_s)$ with $\alpha^+(0) = {{\I}_n}$,  the diagonal entries are of the form $1 + g_{ii}$, and 
off diagonal entries are $g_{ij}$, where $g_{ij} \in {(A_+})_s$ for all $i, j$. 
We choose $b\in (s)$ such that $b=s^l$ with $l\gg 0$, so that $b$ can dilute 
the denominator of each entries. Then $\alpha_s^+(b) \in {\es}(n, A)$.
\end{proof}

The next lemma gives some structural information about commutators.
\begin{lm} \label{lemm 21} Let $\alpha , \beta \in {\es}(n, A)$ and  $A_0$ be a commutative semilocal ring. 
Then the commutator subgroup $$ [\alpha, \beta] \in [\alpha \alpha_0^{-1}, \beta \beta_0^{-1}]\;{\E}(n, A).$$
\end{lm}

\begin{proof}
Since $A_0$ is semilocal, we have ${\es}(n, A_0) = {\E}(n, A_0)$ by Lemma \ref{lemm 26}. Hence $\alpha^+(0), 
\beta^+(0) \in {\E}(n, A_0)$. Let $a = \alpha {\alpha^+(0)}^{-1}$ and $b = \beta {\beta^+(0)}^{-1}$. Then 
\begin{align*}
    [\alpha, \beta ] & =  [\alpha {\alpha^+(0)}^{-1}\alpha^+(0), \beta {\beta^+(0)}^{-1} \beta^+(0) ]\\
         &  = a \alpha^+(0) b \beta^+(0) {\alpha^+(0)}^{-1} a^{-1} {\beta^+(0)}^{-1} b^{-1}\\
& = \big(aba\!^{-1}b\!^{-1}\big) \big(bab\!^{-1}\alpha^+(0) ba\!^{-1}b\!^{-1} \big)
\big( ba \beta^+(0) {\alpha^+(0)}\!^{-1}a\!^{-1}b^{-1} \big) \big(b {\beta^+(0)}^{-1}b\!^{-1} \big).
\end{align*}
Since ${\E}(n, A)$ is a normal subgroup of ${\es}(n, A)$, the elements $bab^{-1}\alpha^+(0) ba^{-1}b^{-1}$, 
$ba \beta^+(0) {\alpha^+(0)}^{-1}a^{-1}b^{-1}$ and $ b {\beta^+(0)}^{-1}b^{-1}$ are in ${\E}(n, A).$ 

\end{proof}

\begin{co} \label{cor 21} Let  $\alpha \in [{\es}(n, A), {\es}(n, A)]$ with $\alpha^+(0) = {\I}_n,$ and 
let $A_0$ be a semilocal commutative ring. Then using the normality property of the elementary $($resp. elementary symplectic$)$ subgroup 
$\alpha$ can be written as $$\underset{k=1}{\overset{t}\Pi} \; 
[\beta_k, \gamma_k ]\; \epsilon ,$$ for some $t\ge 1$, and $\beta_k, \gamma_k \in {\es}(n, A)$, with $\beta_{k}^+(0) = \gamma_{k}^+(0) = 
{\I}_n$, and $\epsilon \in {\E}(n, A)$ with $\epsilon^+(0) = {\I}_n$.
\end{co}

\begin{proof}
Since $\alpha \in [{\es}(n, A), {\es}(n, A)]$, $\alpha = \underset{k=1}{\overset{t}\Pi} [a_k, b_k]$ for some $t\ge 1$. 
Using Lemma \ref{lemm 21} we identify $\beta_k$ with $a_k {{a_k}^+(0)}^{-1}$ and 
$\gamma_k$  with $b_k{{b_k}^+(0)}^{-1}$. This gives $$ \alpha = \underset{k=1}{\overset{t}\Pi}\; [\beta_k, \gamma_k ]\; \epsilon.$$
Then it follows that $\epsilon^+(0) = {\I}_n$, as $\alpha^+(0) = {\I}_n$.
\end{proof}

The next lemma gives a variant of ``Dilation Lemma'' (Lemma \ref{lemm 24}) 
\underline{ for ${\es}(n, A)$}; the special linear (resp. symplectic) group.

\begin{lm} \label{lemm 22} If $\alpha \in {\es}(n, A_s)$ with $\alpha^+(0) = {{\I}_n}$, then $$\alpha^+(b + d){\alpha^+(d)}^{-1} \in {\es}(n, A),$$ 
where $s$ is a non-zero divisor and $b, d\in A_0$ with $b = s^i$ for some $i \gg 0$.
\end{lm}

\begin{proof}
Let $\alpha^+(X) \in {\G}(n, A[X])$, then $\alpha_s^+(X) \in {\E}(n, A_s[X])$. 
Let $\beta^+(X) = \alpha^+(X + d) {\alpha^+(d)}^{-1}$. Then $\beta_s^+(X) \in {\es}(n, A_s[X])$. 
Hence by Lemma \ref{lemm 24}, $\beta^+(bX) \in {\es}(n, A)$. Putting $X = 1$, 
we get  result.
\end{proof}

The following lemma makes use of Lemma \ref{lemm 24} to deduce ``Dilation Lemma'' for the commutator subgroup {$[{\es}(n, A), {\es}(n, A)]$}.
\begin{lm} \label{lemm 23} Let $\alpha_s = \underset{i=1}{\overset{t}\Pi} [{\beta_i}_s, {\gamma_i}_s] \epsilon_s$ for some $t\ge 1$, 
such that ${\beta_i}_s$ and ${\gamma_i}_s \in {\es}(n, A_s)$ and $\epsilon_s \in {\E}(n, A_s)$ 
with $\gamma_{i_{s}}^+(0) = \beta_{i_{s}}^+(0) = {\epsilon_s}^+(0) = {\I}_n$. 
Then $($by identifying with its pullback by using Lemma \ref{lemm 13}$)$ 
$$\alpha_s^+(b+d)\alpha_s^+(d)^{-1} \in [{\es}(n, A), {\es}(n, A)],$$ 
where $b, d \in A_0$ with $b=s^l$ for some $l\gg 0$.
\end{lm}

\begin{proof} Without loss of generality (as we would conclude by the end of the proof),
we can assume $$\alpha_s = [\beta_s, \gamma_s] \epsilon_s = \beta_s \gamma_s {\beta_s}^{-1} {\gamma_s}^{-1} \epsilon_s.$$
Since $\beta_{s}^+(0) = {\I}_n$, by Lemma \ref{lemm 24} ${\beta_s}^+(b), {\gamma_s}^+(b) \in {\es}(n, A)$ for some $b=s^l$ with $l\gg 0$. Also,  
${\epsilon_s}^+(b) \in {\E}(n, A)$ by (\ref{Prop 14}) of Theorem \ref{thm 22}.
Hence

\begin{align*}
\alpha_s^+(b+d){\alpha_s^+(d)}^{-1}   = &  \; \beta_s^+(b+d) \gamma_s^+(b+d) {\beta_s^+(b+d)}^{-1} {\gamma_s^+(b+d)}^{-1}  \\ 
& \epsilon_s^+(b+d) {\epsilon_s^+(d)}^{-1} \gamma_s^+(d)\beta_s^+(d) {\gamma_s^+(d)}^{-1} {\beta_s^+(d)}^{-1}. 
\end{align*}
Since ${\E}(n, A_s)$ and ${\es}(n, A_s)$ are normal subgroups in ${\G}(n, A_s)$, by rearranging (using $\epsilon_s^+(b+d){\epsilon_s^+(d)}^{-1}$ 
as intermediary) we can consider 
$\gamma_s^+(b+d) {\gamma_s^+(d)}^{-1}$ and $ \beta_s^+(b+d){\beta_s^+(d)}^{-1}$ together. Now  by using Lemma \ref{lemm 22} 
we get $\gamma_s^+(b+d) {\gamma_s^+(d)}^{-1}$ and $ \beta_s^+(b+d){\beta_s^+(d)}^{-1} \in {\es}(n, A)$. 
Hence $($by identifying with its pullback by using Lemma \ref{lemm 13}$)$  it follows that $$\alpha_s^+(b + d) {\alpha_s^+(d)}^{-1} \in [{\es}(n, A), {\es}(n, A)].$$
\end{proof}

Now we deduce the graded version of the Local-Global Principle for the commutators subgroups.
\begin{tr}\label{thm 21} Let $\alpha \in {\es}(n, A)$ with $\alpha^+(0) = {\I}_n$. 
If $$\alpha_\mathfrak{p} \in [{\es}(n, A_\mathfrak{p}), {\es}(n, A_\mathfrak{p})] \text{ for all } \mathfrak{p} \in {\rm Spec (A_0)},$$
then $$\alpha \in [{\es}(n, A), {\es}(n, A)].$$ 
\end{tr}

\begin{proof} Since $\alpha_\mathfrak{p} \in [{\es}(n, A_\mathfrak{p}), {\es}(n, A_\mathfrak{p})] \subset [{\es}(n, A_\mathfrak{m}), {\es}(n, A_\mathfrak{m})]$, 
we have for a maximal ideal $\mathfrak{m} \supseteq  \mathfrak{p}$,  $s \in A_0 - \mathfrak{m}$ 
such that $$\alpha_s \in [{\es}(n, A_s), {\es}(n, A_s)],$$  hence $\alpha_s$ can be decomposed as 
$$\alpha_s = \underset{i=1}{\overset{t}\Pi} [{\beta_{i_s}}, {\gamma_{i_s}}]\epsilon_s$$ by Corollary 
\ref{cor 21}, where $\beta_{i_s}, \gamma_{i_s} \in {\es}(n, A_s)$ and $\epsilon_s \in {\E}(n, A_s)$.
Now using the ``Dilation Lemma'' (Lemma \ref{lemm 24} for ${\es}(n, A_s)$ and theorem \ref{thm 22} $(4)$ for 
${\E}(n, A_s)$) on each of these elements, we have for $b = s^i, i \gg 0$ $$\alpha_s^+(b) = 
\underset{i=1}{\overset{t}\Pi} [\beta_{i_s}^+(b), \gamma_{i_s}^+(b)] \epsilon_s^+(b) \in [{\es}(n, A), {\es}(n,A)]E(n, A).$$ 
Since ${\E}(n, A_s) \subseteq [{\es}(n, A_s), {\es}(n, A_s)]$, we have $\alpha_s \in [{\es}(n, A_s), {\es}(n, A_s)]$. 
Let $s_1, \ldots, s_r \in A_0$ non-zero divisors such that $s_i \in A_0 - \mathfrak{m}_i$ and $\langle s_1, \ldots, s_r \rangle = A_0$. Then 
it follows that $b_1 + \cdots + b_r = 1$ for suitable (as before) $b_i \in (s_i)$; and $i=1,\ldots r$. Now
\begin{align*} \alpha_{s_1s_2\ldots s_r} = \; & \alpha_{s_1s_2\ldots s_r}^{+}(1)
 \\
 = \; & \alpha_{s_1s_2\ldots s_r}^+(b_1 + b_2 + \cdots + b_r) \\ = 
 \; & {{\big((\alpha_{s_1})}_{s_2s_3s_4\ldots}}\big)^+(b_1 + \cdots+ b_r) 
 {{\big((\alpha_{s_1})}_{s_2s_3s_4\ldots}}\big)^+(b_2 + \cdots + b_r)^{-1} \!\!\cdots \\
  & {\big((\alpha_{s_i})}_{s_1 \ldots \hat{s_i} \ldots s_r}\big)^{+}\!\!(b_i + \cdots + b_r){\big((\alpha_{s_i})}_{s_1 \ldots 
  \hat{s_i} \ldots s_r}\big)^{+}\!\!(b_{i+1} + \cdots + b_r)^{-1} \\
 & \cdots{{\big((\alpha_{s_r})}_{s_1s_2\ldots s_{r-1}}}\big)^{+}\!\!(b_r){{\big((\alpha_{s_r})}_{s_1s_2\ldots s_{r-1}}}\big)^{+}\!\!(0)^{-1}.
\end{align*} 
Using Lemma \ref{lemm 13} the product is well defined (see Remark \ref{rmk 001}) and using Lemma \ref{lemm 23}, we conclude that $\alpha \in [{\es}(n, A), {\es}(n, A)].$
\end{proof}

\section{Auxiliary Result for Transvection subgroup}
Let $R$ be a commutative ring with $1$. We recall that a finitely generated projective $R$-module $Q$ has a unimodular element 
if their exists  $q \in Q$ such that $qR \simeq R$ and $Q \simeq qR \oplus P$ for some projective $R$-module $P$.

In this section, we shall consider three types of (finitely generated) classical modules;  {\it viz.} projective, 
symplectic and orthogonal modules over graded rings. Let $A = \bigoplus_{i = 0}^{\infty} {A_i}$ be a commutative graded Noetherian ring with identity. 
For definitions and related facts we refer 
\cite{BBR}; \$1, \$2. In that paper, the first author  with A. Bak and R.A. Rao has established an
analogous Local-Global Principle for the elementary transvection subgroup of the
automorphism group of projective, symplectic and orthogonal modules of global
rank at least 1 and local rank at least 3. In this article we deduce an analogous statement for the above classical groups over graded rings. 
We shall assume for every maximal ideal $\mathfrak{m}$ of $A_0$, the  symplectic and orthogonal module $Q_\mathfrak{m} \simeq (A_0^{2m})_{\mf{m}}$ 
with the standard form (for suitable integer $m$).

\begin{re} {\rm By definition the global rank or simply rank of a finitely generated projective $R$-module 
(resp. symplectic or orthogonal $R$-module) is the largest integer $k$ such that $\oplus^{k}R$ 
(resp. $\perp^k$ $\mathbb{H}(R)$) is a direct summand (resp. orthogonal summand) of the module. Here
$\mathbb{H}(R)$ denotes the hyperbolic plane.}
\end{re}
Let $Q$ denote a projective, symplectic or orthogonal $A_0$-module of global rank 
$\ge 1$, and total (or local) 
rank $r + 1$ in the linear case and $2r+ 2$ otherwise and let $Q_1 = Q \otimes_{A_0} A$.\vp

We use the following  notations to deal with the above three classical modules uniformly.
\begin{enumerate}
\item ${\G}(Q_1)$ := the full automorphism group of $Q_1$. \vp 
\item ${\T}(Q_1)$ := the subgroup generated by transvections of $Q_1$.\vp 
\item ${\ET}(Q_1)$ := the subgroup generated by elementary-transvections of $Q_1$.\vp
\item $\widetilde{\E}(r, A)$ := ${\E}(r+1, A)$ for linear case and ${\E}(2r+2,A)$ otherwise.
\end{enumerate} \vp

In \cite{BBR}, the first author with Ravi A. Rao and A. Bak established the ``Dilation Lemma'' and the ``Local-Global Principle'' for 
the transvection subgroups of the automorphism 
group over polynomial rings. In this section, we generalize their results over graded rings. For the statements over polynomial rings 
we request the reader to look at Proposition 3.1 and Theorem 3.6 in \cite{BBR}. \vp

\iffalse
\begin{tr}[Dilation Lemma; {\it cf.} \cite{BBR}, Proposition 3.1] Let $Q[X]=Q\otimes R[X]$ in the linear case, and  be a projective $R$-module. Let $s$ be a 
non-nilpotent element in $R$. Let $\alpha(X) \in G(R[X])$ with $\alpha(0) = {\I}_n$. Suppose

\[ \alpha_s(X) \in
\begin{cases}
{\E}\big(r + 1, R_s[X]\big) & \text{ in the linear case, }\\
{\E}\big(2r + 2, R_s[X]\big) & \text{ otherwise. }
\end{cases}
\]
Then there exists $\widetilde{\alpha}(X) \in {\rm ET}(Q[X])$ and $l \gg 0$ such that $\widetilde{\alpha}(X)$ 
localizes at $\alpha(bX)$ for $b = (s^l)$ and $\widetilde{\alpha}(0) = {\I}_n$.

\end{tr}

\begin{tr}[Local-Global Principle; {\it cf.} \cite{BBR}, Theorem 3.6] Let $Q$ be a projective $R$-module. 
Suppose $\sigma(X) \in G(Q[X])$ with $\sigma(0) = {\I}_n$. If 

\[ \sigma_\mathfrak{m}(X) \in
\begin{cases}
{\E}\big(r + 1, R_\mathfrak{m}[X]\big) & \text{ in the linear case, }\\
{\E}\big(2r + 2, R_\mathfrak{m}[X]\big) & \text{ otherwise. }
\end{cases}
\]
for all $\mathfrak{m} \in $ ${\rm max}\; (R)$, then $\sigma(X) \in {\rm ET \big(R[X]\big)}$.
\end{tr}
\fi

Following is the graded version of the  ``Dilation Lemma''.

\begin{pr}[Dilation Lemma]\label{thm 41} Let $Q$ be a projective $A_0$-module, and $Q_1 = Q \otimes_{A_0} A$. 
Let $s$ be a non-nilpotent element in $A_0$. Let $\alpha \in G(Q_1)$ with $\alpha^+(0) = {\I}_n$. Suppose
\[ \alpha_s \in
\begin{cases}
{\E}\big(r + 1, (A_0)_s\big) & \text{ in the linear case,}\\
{\E}\big(2r + 2, (A_0)_s\big) & \text{ otherwise. }
\end{cases}
 \]
Then there exists $\widetilde{\alpha} \in {\rm ET}(Q_1)$ and $l \gg 0$ such that $\widetilde{\alpha}$ 
localizes at $\alpha^+(b)$ for some $b = (s^l)$ and $\widetilde{\alpha}^+(0) = {\I}_n$.

\end{pr}

\begin{proof}. Let $\alpha \in {\rm G}(Q_1)$ and $\alpha(X) := \alpha^+(X) \in {\rm G}(Q_1[X])$. 
Since $\alpha_s \in \widetilde{{\E}}(r, A_s)$, $\alpha_s(X) \in \widetilde{{\E}}(r, A_s[X])$. Since $\alpha(0) := 
\alpha^+(0) = {\I}_n$, we can apply ``Dilation Lemma'' for the ring $A[X]$ ({\it cf.}\cite{BBR}, Proposition 3.1) and hence there exists 
$\widetilde{\alpha}(X) \in {\ET}(Q_1[X])$ such that $\widetilde{\alpha}_s(X) = \alpha_s^+(bX)$ for some
$b\in (s^l)$, $l \gg 0$. Substituting $X = 1$, we get $\widetilde{\alpha} \in {\rm ET}(Q_1)$ and 
$\widetilde{\alpha}_s := \widetilde{\alpha}_s(1) = \alpha_s(b) = \alpha_s^+(b)$. This proves the proposition.  
\end{proof}

\begin{lm} If $\alpha \in {\G}(Q_1)$ with $\alpha^+(0) = {\I}_n$ such that $\alpha_s \in {\E}(n, A_s)$, 
then $\alpha^+(b + d)\alpha^+(d)^{-1} \in {\rm ET}(Q_1)$ for some $b\in (s^l)$ with $l \gg 0$, and $d\in A_0$.
\end{lm}

\begin{proof}
Let $\beta = \alpha_s^+(b + d){\alpha_s^+(d)}^{-1}$. Since $\alpha_s \in {\E}(n, A_s)$ 
implies $\alpha^+(a) \in {\E}(n, A_s)$ for $a \in A_s$ and hence $\beta \in {\E}(n, A_s)$. 
Hence by Theorem \ref{thm 41}, their exists a $\widetilde{\beta} \in {\ET}(Q_1)$ such that $\widetilde{\beta}_s^+(b) = \beta$. Hence the lemma follows. 
\end{proof}

\begin{tr}[Local-Global Principle]\label{thm 51} Let $Q$ be a projective $A_0$-module, and $Q_1 = Q \otimes_{A_0} A$. 
Suppose $\alpha \in G(Q_1)$ with $\alpha^+(0) = {\I}_n$. If 

\[ \alpha_\mathfrak{m} \in
\begin{cases}
{\E}\big(r + 1, (A_0)_\mathfrak{m}\big) & \text{ in the linear case, }\\
{\E}\big(2r + 2, (A_0)_\mathfrak{m}\big) & \text{ otherwise. }
\end{cases}
\]
for all $\mathfrak{m} \in $ ${\rm max}\; (A_0)$, then $\alpha \in {\rm ET \big(Q_1\big)}\subseteq {\rm T \big(Q_1\big)}$.
\end{tr}

\begin{proof} Let $\alpha \in {\rm G}(Q_1)$ with $\alpha^+(0) = {\I}_n$. Since $\alpha_\mathfrak{m} \in \widetilde{{\E}}(r, A_\mathfrak{m})$. 
Hence their exists a non-nilpotent $s \in A_0 - \mathfrak{m}$, such that $\alpha_s \in \widetilde{{\E}}(r, A_s)$. Hence by the above ``Dilation Lemma'' 
there exists $\widetilde{\alpha} \in {\rm ET}(Q_1)$ such that $\widetilde{\alpha}_s = \alpha_s^+(b)$, for some $b = s^l$ with $l \gg 0$. 
For each maximal ideal $\mathfrak{m}_i$ we can find a suitable $b_i$. Since $A$ is Noetherian, it follows that $b_1 + \cdots + b_r = 1$ for some positive integer $r$. 
Now we observe that $\alpha^+(b_i + \cdots + b_r){\alpha^+(b_{i+1} + \cdots + b_r)}^{-1} \in {\rm ET}(Q_1)$ and hence  
calculating in the similar manner as we did it in Remark \ref{rmk 001}, we get 

\begin{align*}
\alpha  = \alpha^+(1) = \;& \alpha^+(b_1 + \cdots + b_r) {\alpha^+(b_2 + \cdots + b_r)}^{-1} \cdots \alpha^+
(b_i + \cdots + b_r) \\ & {\alpha^+(b_{i+1} + \cdots + b_r)}^{-1} \cdots \alpha^+(0)^{-1}\in {\rm ET}(Q_1)\subseteq {\rm T}(Q_1); 
\end{align*}
as desired.
\end{proof}

\section*{Acknowledgement}

The second author expresses his sincere thanks to IISER Pune for allowing him to work on this project and 
providing him their infrastructural facilities.

\medskip

\noindent
{\it Indian Institute of Science Education and Research $($IISER Pune$)$,  India},\\ 
{\it email: rabeya.basu@gmail.com, rbasu@iiserpune.ac.in}\vp \newline
{\it Indian Institute of Science Education and Research $($IISER Pune$)$,  India},\\
{\it email: manishks10100@gmail.com}\vp 
\medskip

\end{document}